\documentclass{amsart}
\usepackage[all]{xy}

\numberwithin{equation}{section}
\theoremstyle{plain}
\newtheorem*{theorem*}{theorem}
\newtheorem*{lemma*}{Lemma}
\newtheorem{maintheorem}{Theorem}
\newtheorem{mainprop}[maintheorem]{Proposition}

\theoremstyle{definition}
\newtheorem{definition*}{Definition}

\usepackage[T1]{fontenc}

\newcommand{\bl}[2]{\left\langle #1,#2\right\rangle}

\DeclareMathOperator{\coh}{coh}

\renewcommand{\d}[1]{\mathbb{#1}}

\newcommand{\define}{\stackrel{\operatorname{def}}{=}}

\newcommand{\ds}{\oplus}

\newcommand{\epi}{\xymatrix{{}\ar@{->>}[r]&{}}}
\DeclareMathOperator{\End}{End}

\newcommand{\fun}{\mapsto}

\DeclareMathOperator{\Hom}{Hom}

\DeclareMathOperator{\im}{im}

\def\mod{\operatorname{mod}}
\newcommand{\mor}{\longrightarrow}

\newcommand{\T}{\mathcal{T}}
\renewcommand{\O}{\mathcal{O}}

\DeclareMathOperator{\Num}{Num}

\DeclareMathOperator{\Pic}{Pic}

\renewcommand{\r}[1]{\mathcal{#1}}
\DeclareMathOperator{\rk}{rk}

\renewcommand{\r}[1]{\mathcal{#1}}
\DeclareMathOperator{\num}{num}
\DeclareMathOperator{\rad}{rad}

\newcommand{\tr}{\otimes}

\newcommand{\Q}{\mathbb{Q}}

\newcommand{\Z}{\mathbb{Z}}

\newtheorem{lemma}{Lemma}[section]
\newtheorem{proposition}[lemma]{Proposition}

\newtheorem{lemmas}{Lemma}[subsection]
\newtheorem{propositions}[lemmas]{Proposition}

\newtheorem{corollarys}[lemmas]{Corollary}

\theoremstyle{definition}

\newtheorem{definition}[lemma]{Definition}

\newtheorem{definitions}[lemmas]{Definition}

\theoremstyle{remark}

\newtheorem{remarks}[lemmas]{Remark}

\author{Louis de Thanhoffer de Volcsey}
\email{louis.dethanhofferdevolcsey@utoronto.ca}
\address{Department of Computer and Mathematical Sciences, University of Toronto at Scarborough\\
Toronto, Ontario\\ 
Canada\\ 
M1C 1A4}
\author{Michel Van den Bergh}
\email{michel.vandenbergh@uhasselt.be}
\address{University of Hasselt\\Martelarenlaan 42\\3500 Hasselt\\Belgium}
\title{On an analogue of the Markov equation for exceptional collections of length 4}
\thanks{The first author is a  Postdoctoral Fellow at the UTSC}
\thanks{The second author is a director of research at the Research Foundation Flanders (FWO)}
\thanks{While working this paper the second author was supported by the FWO projects
1503512N: ``Non-commutative algebraic geometry'' and
G0D8616N:  ``Hochschild cohomology and deformation theory of triangulated categories''}
\begin{document}
\begin{abstract}
  We classify the solutions to a system of equations, introduced by
  Bondal, which encode numerical constraints on full exceptional
  collections of length 4 on surfaces.  The corresponding result for
  length 3
  is well-known and states that there is essentially one solution,  namely  the one corresponding to the standard exceptional collection on the surface $\d{P}^2$. This was essentially proven by Markov in 1879 (see \cite{Markov_1879}).

  It turns out that in the length 4 case, there is one special
  solution which corresponds to $\mathbb{P}^1\times\mathbb{P}^1$
  whereas the  other solutions are obtained from $\mathbb{P}^2$ by a procedure we call \emph{numerical blowup}. Among these solutions, three  are of geometric origin ($\mathbb{P}^2\cup \{\bullet\}$, $\mathbb{P}^1\times\mathbb{P}^1$ and the ordinary blowup of $\mathbb{P}^2$ at a point).
  The other solutions are parametrized by $\mathbb{N}$  and very likely
  do not correspond to commutative surfaces. However they can be
  realized as noncommutative surfaces, as was recently shown by Dennis
  Presotto and the first author in \cite{deThanhPres_2}.
\end{abstract}
\maketitle
%
\section{Introduction and statement of results}
\label{sec:introduction}
Below $k$ is an algebraically closed field. 
All objects and categories we consider are $k$-linear. 
For a general triangulated $\Hom$-finite category $\T$, 
 the Grothendieck group $K(\T)$ of $\T$ is
equipped with a bilinear form (the ``Euler form'') defined by the formula
$\bl{[F]}{[G]}:=\sum_i (-1)^i\dim_k \Hom_\T(F,G[i])
$. 
If $\T$ has a Serre functor $S$ in the sense of Bondal and Kapranov \cite{Bondal4} then this yields an automorphism $s$ on
$K(\T)$ satisfying the formula $\bl{v}{sw}=\bl{w}{v}$. In particular
the left and right radical of $\langle-,-\rangle$ coincide and we may define the numerical Grothendieck groupof $\T$ as
$K(\T)_{\num}\define K(\T)/\rad{\bl{-}{-}}$. 
If $\T=D^b(\coh(X))$ for a smooth projective  variety $X$ of dimension $d$ then
$K(X)_{\num}\define K(\T)_{\num}$ is a finitely generated  free abelian group. 
Moreover, the action of  
$(-1)^d s$ on $K(\T)$ and hence on $K(\T)_{\num}$ is unipotent \cite[Lemma 3.1]{Bondal_Polishchuk_93}.

A full exceptional collection $(E_i)_i$ in $\T$ defines a basis $(e_i)_i$ of $K(\T)$ for which the Gram matrix $M:=\langle e_i,e_j\rangle_{ij}$ is upper triangular with $1$'s on the diagonal.
It is clear that in this case we have $K(\T)=K(\T)_{\num}$. We will call an arbitrary basis $(e_i)_i$ of $K(\T)$ with such a Gram matrix \emph{exceptional}. The braid group acts by mutation on exceptional bases  and hence on the corresponding Gram matrices. This can be extended to an action of the \emph{signed} braid group in an obvious way.

\medskip

The previous discussion naturally leads to the problem of classfying finitely generated free abelian groups~$K$ 
equipped with a non-degenerate bilinear form $\langle-,-\rangle$\footnote{we refer to these as \emph{lattices}}, a corresponding Serre automorphism $s$ such that
that $\pm s$ acts unipotently and an exceptional basis $(e_i)_i$. It is this problem that we discuss
in this paper.  We will however restrict ourselves to the case where $s$ acts unipotently 
as our interest is numerical restrictions on surfaces.

In  \cite[Example 3.2]{Bondal_Polishchuk_93} (see Lemma \ref{suni1} below), it was shown that if $\rk K=3$  then the unipotency of $s$ implies that the coefficients of the Gram matrix
\[
M=\begin{bmatrix}
1&a&b\\
0&1&c\\
0&0&1
\end{bmatrix}
\]
must satisfy the \emph{Markov equation}
\begin{equation}
\label{eq:rk3}
a^2+b^2+c^2-abc=0
\end{equation}
It was shown by Markov  \cite{Markov_1879} that all solutions to this equation
may be obtained by a kind of mutation procedure starting from the 
basic solution $(3,3,3)$. This procedure turns out to correspond 
to the mutation of exceptional bases.
In this way one obtains  $K\cong K(\d{P}^2)$.\\[\medskipamount]

Similarly, if $\rk(K)=4$, we can write the Gram matrix as 
\begin{equation}
\label{eq:rk4m}
M=
\begin{bmatrix}
1&a&b&c\\
0&1&d&e\\
0&0&1&f\\
0&0&0&1
\end{bmatrix}
\end{equation}
In this case, the unipotency of $s$  yields a pair of Diophantine equations \cite{BondalSymplectic} (see Lemma \ref{suni2} below)
\begin{equation}
\label{eq:rk4}
\begin{cases}
acdf-abd-ace-bcf-def+a^2+b^2+c^2+d^2+e^2+f^2=0\\
af-be+cd=0
\end{cases}
\end{equation}
The following is our main result (see \S\ref{sec:rk4} below).
\begin{maintheorem} \label{mainth} Let $M$ be a solution of \eqref{eq:rk4}. Then under the action of the signed braid group
 $M$ is equivalent to exactly
 one of the following solutions
\[
\begin{bmatrix}
1&2&2&4\\
0&1&0&2\\
0&0&1&2\\
0&0&0&1
\end{bmatrix}
\qquad\qquad\qquad
\begin{bmatrix}
1&n&2n&n\\
0&1&3&3\\
0&0&1&3\\
0&0&0&1
\end{bmatrix}  \qquad \text{for $n\in \mathbb{N}$}
\]
\end{maintheorem}
The first solution corresponds to $\mathbb{P}^1\times\mathbb{P}^1$ with
its standard exceptional collection \[\bigg(\mathcal{O}(0,0),\mathcal{O}(1,0),
\mathcal{O}(0,1),\mathcal{O}(1,1)\bigg).\]

For $n=0,1$ the solutions in the second family correspond respectively to
$\d{P}^2\cup \{\bullet\}$ and the first
Hirzebruch surface $\d{F}_1$. For other values of $n$ it is easy
to see that the solutions cannot be realized by a full exceptional sequence on a rational surface and in fact presumably cannot be realized on any smooth projective surface. However they
do arise as Grothendieck groups of \emph{noncommutative surfaces} \cite{DennisPresottoInpreparation}.
Moreover for $n=2$ a mutation 
equivalent solution (see \S\ref{eq:single} below) given by
\[
\begin{bmatrix}
1&2&1&5\\
0&1&0&4\\
0&0&1&2\\
0&0&0&1
\end{bmatrix}
\]
was realized as a different
noncommutative surface in \cite{deThanhPres_2}.

\medskip

Our proof of Theorem \ref{mainth} is inspired by Rudakov's 
approach \cite{Rudakov} to the classification of full exceptional collections on quadrics. To this end we discuss some algebro-geometric concepts in the context of triples $(K,\langle-,-\rangle,s)$. We hope this will be of independent interest. 
Our main source of inspiration 
is the following result. 
\begin{mainprop}(Proposition \ref{lem:rank} below, see also \cite{BelmansRaedschelders}) \label{prop:mainprop}
Let $X$ be a  smooth projective surface. Then the action of $s$ on $K(X)_{\num}$ satisfies 
$
\rk(s-1)\le 2
$.
\end{mainprop}
Assume $\rk K$ is arbitrary. We say that $K$ is of surface* type if $\rk(s-1)\le 2$ and $(s-1)^2\neq 0$. In that case we 
define a 3-step filtration $K=F^0K\supset F^1K\supset F^2K\supset 0$
with $F^1K=(\ker(s-1)^2)_{\Q}\cap K$, $F^2K=(\im(s-1)^2)_{\Q}\cap K$
which
serves as a substitute for the codimension
filtration on $K(X)$.  Then $\Num(K)\define F^1K/F^2K$ is a free
abelian group which is a substitute for the numerical Picard group
of~$X$. In particular $-\langle-,-\rangle$ restricts to a symmetric
nondegenerate bilinear ``intersection form'' $(-,-)$ on $\Num(K)$.
Moreover $\Num(K)$ contains a distinguished element $\omega$, well defined up to sign, which serves as a substitute for
the canonical class. Using these ingredients we
develop some rudimentary numerical algebraic geometry for~$K$. In particular
we define a numerical notion of blowup which is more general than
the geometric notion. The lattices in the second family in
the statement of Theorem \ref{mainth} are obtained by numerical
blowup of $K(\mathbb{P}^2)$.

The quantity $\delta(K)\define (\omega,\omega)$ is an invariant of $K$ which
we call the degree of $K$. One
computes that the lattices appearing in Theorem \ref{mainth} have
$\delta(K)$ equal to $8$ resp. $9-n^2$ (see \S\ref{sec:blowupP2} below).

\medskip 

We believe it would be interesting to extend the results in this
paper to higher rank lattices. It seems likely that the methods in \cite{Perling}
would be helpful here. Note however that at least in the rank 4 case
we could avoid preimposing the  rationality conditions
which are used in \cite{Perling}.
\section{Preliminaries}\label{sec:geometricconditions}
Below a ``variety'' is automatically connected. The same goes for a ``surface''.
Let $X$ be a smooth projective variety of dimension $d$ over $k$. 
We recycle the associated notation $K(X),s$ and $\langle-,-\rangle$ introduced in the beginning of 
the introduction.
The following property of $s$ was already mentioned:
\begin{proposition} \cite[Lemma 3.1]{Bondal_Polishchuk_93} \label{lem:unipotent}  The
action of $s$ on $K(X)$ satisfies \[((-1)^ds-1)^{d+1}= 0\].
\end{proposition}
\begin{proof}
The idea of the proof goes as follows: let $(F^iK(X))_i$ be the codimension filtration. Then $F^{d+1}K(X)=0$ and one shows that $((-1)^ds-1)F^iK(X)\subset F^{i+1}K(X)$. 
\end{proof}
From the proof it follows in particular that $\im((-1)^ds-1)^d\subset F^d K(X)$. Note that there are canonical maps $\rk: K(X)\rightarrow\mathbb{Z}$, $\int:F^d K(X)\rightarrow \mathbb{Z}$. For use below we record the following formula:
\begin{lemma} \label{lem:dpart}
Let $\r{E}\in K(X)$.  Then
\[
\int ((-1)^d s-1)^d(\r{E})= \rk(\r{E}) \int c_1(\omega_X)^d
\]
\end{lemma}
\begin{proof} This is a straightforward application of the Grothendieck Riemann-Roch theorem.
\end{proof}
It will be convenient to put $\delta(X)\define c_1(\omega_X)^d$. We will refer to $\delta(X)$ as the \emph{degree} of $X$. 
We now give the proof of Proposition \ref{prop:mainprop}
\begin{proposition}\label{lem:rank} Assume that  $X$ be a smooth projective surface. Then the action of $s$ on 
$K(X)_\text{num}$ satisfies
$\rk(s-1)\le 2$.  
\end{proposition}
\begin{proof}
This result appeared in the first author's Ph.D thesis where it was proven through an argument using cohomology. After learning this result the authors of
\cite{BelmansRaedschelders} found an independent proof, which appeared in loc.\ cit.\ Here we give yet another proof.\\

Let $V\define K(X)_{\num,\Q}$. The codimension filtration on $V$ satisfies $\dim V/F^1V=1$ and  $\dim F^2 V=1$. Choose $o\in V$,
a representative for $V/F^1V$. Then 
\[
(s-1)(V)=k(so-o)+(s-1)(F^1V)\subset k(so-o)+F^2V
\] 
This proves that indeed $\dim (s-1)(V)\le 2$.
 \end{proof}
We recal the following result on the relation between
the Euler form and the intersection form.
\begin{lemma}\label{lem:eulerform-picardform}
Let $X$ be a  smooth  projective surface.
For line bundles $\r{L}$ and $\r{L}'$ on $X$, we have
\begin{equation}
\label{eq:standard2}
c_1(\r{L})\cdot c_1(\r{L}')=-\bl{[\r{L}]-[\r{O}_X]}{[\r{L'}]-[\r{O}_X]}
\end{equation}
where the first product is the intersection pairing on the Picard group.
\end{lemma}
\begin{proof} (Sketch) One shows that both sides are additive in $\r{L}$, $\r{L}'$. This reduces one to the case
$\mathcal{L}=\O(D)$, $\mathcal{L}'=\O(E)$ with $D$, $E$ being transversal. In that case the result is a direct computation
using the definitions.
\end{proof}
\section{Algebraic geometry for numerical Grothendieck groups}
\label{sec:numer}
\subsection{Preliminaries}
We recall some results from \cite{Bondal_Polishchuk_93,BondalSymplectic}.
\begin{definitions}\label{def:lattice}
A \emph{Serre lattice} consists of a finitely generated free abelian group~$K$, a nondegenerate bilinear form $\bl{-}{-}$ on $K$ and an automorphism $s$ of $K$ such that
\begin{equation}
\label{eq:iter}
 \forall v,w, \in K:\bl{v}{sw}=\bl{w}{v}
\end{equation}
\end{definitions}
Iterating \eqref{eq:iter} we see in particular that
\[
\langle sv,sw\rangle=\langle v,w\rangle
\]
The \emph{Gram matrix} of $K$ with respect to a basis $(e_i)_i$ is the matrix $M=\langle e_i,e_j\rangle_{ij}$. It is easy to see that the matrix of $s$ with respect to the same basis is then given by
\begin{equation}
\label{eq:sformula}
s=M^{-1} M^{\mathrm{t}}.
\end{equation}
In particular we see that if $\langle-,-\rangle$ is unimodular (for example if there is an exceptional basis) then the characteristic polynomial of $s$ is given by
\[
\chi(s)(t)\define \sum_i \chi_i(s)t^i=\det(M-tM^{\mathrm{t}})
\]
It follows that $\chi(s)(t)$ satisfies the functional equation
\[
\chi(s) (t^{-1})=(-1)^nt^{-n}\chi(s)(t)
\]
where $n\define \rk K$. In particular $\chi_i(s)=(-1)^n \chi_{n-i}(s)$. It is also clear that $\chi_0(s)=\det(s)=1$ and $\chi_n(s)=(-1)^n$.

\begin{lemmas} \cite{Bondal_Polishchuk_93} 
\label{suni1} If $\rk K=3$ and if $K$ has an exceptional
basis
then
$s$ is unipotent if and only \eqref{eq:rk3} holds.
\end{lemmas}
\begin{proof} 
$s$ is unipotent iff $\chi(s)(t)=(1-t)^3$ which in view of the above discussion is equivalent to $\chi_1(s)=-3$.
One verifies that this yields precisely \eqref{eq:rk3}.
\end{proof}
A similar result hods in the rank 4 case:
\begin{lemmas} \label{suni2}
\cite{BondalSymplectic} If $\rk K=4$ and if $K$ has an exceptional
basis
then
$s$ is unipotent if and only \eqref{eq:rk4} holds.
\end{lemmas}
\begin{proof} 
Now we must have $\chi_1(s)=-4$, $\chi_2(s)=6$.
Let $q_1$, $q_2$ denote  the left-hand side of the respective equations in \eqref{eq:rk4}. Then using a computer algebra system
one checks $\chi_1(s)=q_1-4$. $\chi_2(s)=q_2^2-2q_1+6$. The conclusion follows.
\end{proof}
\subsection{Lattices of surface type}
\begin{definitions}\label{def:SPStype}
Let $K=(K,\langle-,-\rangle,s)$ be a Serre lattice.
We say that $K$ is of \emph{surface type} if
\begin{enumerate}
	\item $s$ is unipotent.
	\item $\rk(s-1)\le 2$.
\end{enumerate}
If in addition one has
\begin{enumerate}
\setcounter{enumi}{2}
\item $(s-1)^2\neq 0$.
\end{enumerate}
then $K$ is said to be of \emph{surface* type}.
\end{definitions}
\begin{propositions}\label{thm:SPStype}
Let $X$ be a  smooth projective surface. Then $K(X)_\text{num}$ is of surface type. Moreover
$K(X)_\text{num}$ is of surface* type if and only if $\delta(X)\neq 0$.
\end{propositions}

\begin{proof}
The first part  follows by combining Propositions \ref{lem:unipotent} and \ref{lem:rank}. The second part
follows from Lemma \ref{lem:dpart}.
\end{proof}
\begin{remarks}
If $X$ is a Calabi-Yau surface then $K(X)_{\num}$ is of surface type but not of surface* type as $s=1$.
\end{remarks}
Below we will often work in the $\Q$-vector space $K_{\Q}\define \Q\tr_{\Z}K$. In that case we extend $s$ and $\langle-,-\rangle$ silently to $K_{\Q}$. This convention is also used
for other concepts introduced below.

\subsection{The numerical codimension filtration}
We now construct a numerical analogue of the codimension filtration. 
\begin{propositions} \label{prop:jordan} Let $K=(K,\langle-,-\rangle,s)$ be of surface* type and $V\define K_\d{Q}$ as above. Then the Jordan blocks of $s\in \End_\d{Q}(V)$ have sizes $(3,1,\ldots,1)$. In
particular $(s-1)^3=0$.
\end{propositions}
\begin{proof} Let $(n_1,\ldots,n_t)$ be the sizes of the Jordan blocks of $s$. The fact that $\rk(s-1)\le 2$ implies $\sum_i (n_i-1)\le 2$.
The fact that $(s-1)^2\neq 0$ implies the existence of at least one $n_i$ such that $n_i\ge 3$.
The conclusion follows.
\end{proof}
\begin{lemmas}\label{filtration}
Assume $K=(K,\langle-,-\rangle,s)$ is of surface* type and put $V=K_{\Q}$.   Put $F^1V\stackrel{\textrm{def}}{=}\ker (s-1)^2$ and
  $F^2V\stackrel{\textrm{def}}{=}\im(s-1)^2$. This yields a filtration: 
\[
0=F^3V\subset F^2V\subset F^1V\subset F^0V=V
\]
such that
\begin{itemize}
\item
$(s-1)F^i V \subset F^{i+1}V$  and moreover $(s-1)F^1V = F^2V$.
\item
$V/F^1V\cong F^2V\cong \Q$.
\end{itemize}
\end{lemmas}
\begin{proof}
This is a straightforward consequence of Proposition~\ref{prop:jordan}.
\end{proof}
This filtration has the following, more intrinsic characterization.
\begin{lemmas}\label{filtrationunique}
Let $K=(K,\langle-,-\rangle,s)$ be a Serre lattice  and let $V=K_{\Q}$.
Let $(\r{F}^iV)_i$ be a filtration 
$$0=\r{F}^3V\subset \r{F}^2V\subset \r{F}^1V\subset \r{F}^0V=V$$
such that
\begin{itemize}
\item $(s-1)\r{F}^iV\subset \r{F}^{i+1}V$
\item $V/\r{F}^1V\cong \r{F}^2V\cong \Q$.
\end{itemize}
then $\r{F}$ coincides with the filtration $F$ on $V$ constructed in Lemma \ref{filtration} whenever $K$ is of surface* type.
\end{lemmas}
\begin{proof}
As the dimensions of $F^iV$ and $\r{F}^iV$ coincide, it suffices to verify the existence of  appropriate inclusions. 
Since $(s-1)^2 \r{F}^1V=0$, we have $\r{F}^1 V  \subset F^1 V$ and in a similar vein $(s-1)^2 V =F^2 V  \subset \r{F}^2V$. 
\end{proof}
\begin{corollarys}\label{canonical=codimension}
Let $X$ be a  smooth projective surface such that $\delta(X)\neq 0$. The codimension filtration on $K(X)_{\num,\Q}$ 
coincides with the filtration defined in Lemma \ref{filtration}, which is well-defined  by Proposition \ref{thm:SPStype}.
\end{corollarys}

\begin{proof} 
The fact that the codimension filtration of a smooth projective variety satisfies the first condition of Lemma \ref{filtrationunique}
is part of the proof of Proposition \ref{lem:unipotent}. The fact that the second condition holds is clear.
\end{proof}
The above result justifies the following definition:
\begin{definitions}
Let $K$ be of surface* type. The filtration on $V=K_{\Q}$ constructed in Lemma \ref{filtration} is called the \emph{codimension filtration}. The induced filtration on $K$ defined by $F^iK\define F^iV\cap K$
will be referred to as the \emph{codimension filtration on $K$}.
\end{definitions}
\subsection{The numerical Picard group} 
\begin{definitions}\label{def:Picard} Let $K=(K,\langle-,-\rangle,s)$ be a lattice of surface* type. 
We define the numerical Picard group of $K$ as $$\Num(K)\stackrel{\textrm{def}}{=}F^1K/F^2K$$
\end{definitions}
Clearly $\Num(K)$ is a free abelian group and $\rk \Num(K)=\rk K-2$. 
\begin{propositions}\label{negativeintersectionform} 
The restriction of $\langle-,-\rangle$ induces a nondegenerate symmetric form on $\Num_\Q(K)\stackrel{\textrm{def}}{=}\Num(K)\tr_\Z \Q$.
\end{propositions}
\begin{proof}
Put $V=K_{\Q}$ and $\Num(V)=F^1V/F^2V$. Then clearly $\Num(V)=\Num_{\Q}(K)$.
We first show that $F^2 V$ lies in the right radical of the restriction of $\bl{-}{-}$ to $F^1V$.
Let $v \in F^1V$ and $(s-1)^2(w) \in F^2V$. Then 
$$\bl{v}{(s-1)^2w}=\bl{s^{-2}(s-1)^2v}{w}=0$$
since $(s-1)^2v=0$. A similar proof shows that $F^2 V$ is also in the left radical, showing that $\langle-,-\rangle$ is indeed well defined on $\Num(V)$.

Let $v,w\in F^1V$. Then $\langle v,w\rangle-\langle w,v\rangle=\langle w,(s-1)v\rangle=0$ since $(s-1)v\in F^2V$ is in the radical of $\langle-,-\rangle$, as shown above.
Thus $\langle-,-\rangle$ is symmetric when restricted
to $F^1V$.

Finally we show that $\langle-,-\rangle$ induces a non-degenerate bilinear form on $\Num(V)$. Since $\langle-,-\rangle$ is non-degenerate on $V$
it follows from the second property in Lemma \ref{filtration}  that $\dim(F^1V)^\perp=1$. Since $F^2V\subset (F^1V)^\perp$ is one-dimensional, also by Lemma \ref{filtration}, we conclude $(F^1V)^\perp=F^2V$.
This yields the desired conclusion.
\end{proof}

\begin{definitions}
\label{def:negativeintesectionform}
The restriction of $-\bl{-}{-}$ to $F^1K$ is called the intersection
form\footnote{the $(-)$ sign is motivated by Lemma \ref{lem:picardformcoincides}} and denoted by $(-,-)$. The induced form on $\Num(K)$ is also denoted by $(-,-)$. Sometimes we write $v\cdot w$ instead of $(v,w)$.
\end{definitions}
\begin{lemmas}
\label{lem:picardformcoincides}
Let $X$ be a  smooth projective surface  such that $\delta(X)\neq 0$. Let $\Num(X)$ be the group of divisors 
on $X$, up to numerical equivalence. Then
the map
$$\Phi: \Num_\Q(X) \mor \Num_{\Q}\left(K(X)_{\num}\right): [\r{L}] \fun [\r{L}]-[\r{O}_X]$$
is an isomorphism of groups such that $\Phi([\r{L}]\cdot [\r{L}'])=\Phi([\r{L}])\cdot\Phi([\r{L'}])$.
\end{lemmas}

\begin{proof}
We denote the classical codimension filtration on $K(X)$ by  $(\r{F}^iK(X))_i$ and we  use the same notation for the induced filtration on $K_{\num}(X)$. It is well known that the morphism 
\[
\Phi: \Pic(X) \stackrel{\simeq}{\mor} \r{F}^1K(X)/\r{F}^2K(X): [\r{L}] \fun [\r{L}]-[\r{O}_X]
\]
is an isomorphism. 
Lemma \ref{lem:eulerform-picardform} shows that $\Phi(\r{L}\cdot \r{L}')=\Phi(\r{L})\cdot \Phi(\r{L'})$. 
This also implies that the radicals of both forms coincide, showing that $\Phi$ descends to an isomorphism
\[
\Phi: \Num_\Q(X)\stackrel{\simeq}{\mor} \bigg( \r{F}^1K(X)/\r{F}^2K(X))/ \rad \bl{-}{-}\bigg)\tr_\Z \Q 
\]
Now by construction we have
\[
(\r{F}^1K(X)/\r{F}^2K(X))/ \rad \bl{-}{-})\otimes_{\Z}\Q= (\r{F}^1K(X)_{\num}/\r{F}^2K(X)_{\num}) \otimes_{\Z}\Q
\] 
Moreover, by Corollary \ref{canonical=codimension}, the latter coincides with the group $\Num_\d{Q}(K(X)_{\num})$ constructed using the filtration in \ref{filtration}.
\end{proof}

\subsection{The canonical class}
We now define
 analogues of  structure sheaves  and canonical sheaves in a lattice of surface* type.
\begin{definitions}\label{def:delta}
Let $K=(K,\langle-,-\rangle,s)$ be of surface* type.
\begin{itemize}
\item A \emph{structure element} in $K$ is an element  ${{o}} \in K$ such that $\overline{{{o}}}$ generates $K/F^1K\cong \Z$. 
\item
The element $\tilde{\omega}\stackrel{\textrm{def}}{=}(s-1)o\in F^1K$ is called the \emph{canonical element} of $K$ associated to ${{o}}$. Its image $\omega$ in $\Num(K)$ is called
the canonical class. Note that by the definition of $s$,  $\langle {{o}},\tilde{\omega}\rangle=\langle o,o\rangle-\langle o,o\rangle =0$.
\item
The \emph{degree} of $K$ is $\delta(K)\stackrel{\textrm{def}}{=}(\omega,\omega)$.
\end{itemize}
\end{definitions}

\begin{lemmas}
\label{lem:deltawd}
Assume that $(K,\langle-,-\rangle,s)$ is of surface* type.
Then the canononical class $\omega\in \Num(K)$ is independent of the choice of ${{o}}$, up to sign.
Hence
$\delta(K)$ is an integer which is independent of the choice of ${{o}}$. Moreover $\delta(K)\neq 0$.
\end{lemmas}
\begin{proof}
Any other element generating $K/F^1K$ must be of the form ${{o}}'\stackrel{\textrm{def}}{=}\pm ({{o}}+\gamma)$ for some $\gamma \in F^1K= \ker(s-1)^2$. If we let $\tilde{\omega}'\stackrel{\textrm{def}}{=}(s-1){{o}}'$, then
$\tilde{\omega}'=\pm \tilde{\omega}\pm (s-1)\gamma$. Since $(s-1)\gamma\in F^2K$ we conclude $\omega=
\pm\omega^\prime$.

Now assume $\delta(K)=0$. Then $\langle (s^{-1}-1)(s-1){{o}},{{o}}\rangle=0$.\\ Since $f\define(s^{-1}-1)(s-1){{o}}\in F^2K$ and $K=\Z {{o}}\oplus F^1K$
we conclude $\langle f,K\rangle=0$. Hence $f=0$ by the non-degeneracy of $\langle-,-\rangle$. However this is impossible since (using the decomposition $K=\Z {{o}}\oplus F^1K$ once more) 
we would then have $(s-1)^2(K)=0$, contradicting the hypothesis that $K$ is of surface* type.
\end{proof}
In view of Lemma \ref{lem:deltawd} the following definition is natural:
\begin{definitions}  If $K$ is of surface type but not of surface* type then we put $\delta(K)=0$.
\end{definitions}

\begin{remarks}
See Lemma \ref{lem:rank-degree} for the relation between $\delta(K)$ and $\delta(X)$.
\end{remarks}

\subsection{Rank and degree functions.}\label{sec:rank}
Let $K=(K,\langle-,-\rangle,s)$ be of surface* type. let ${{o}}$ be a structure element and  $\tilde{\omega} \in F^2(K)$ its associated 
canonical element. We then have a decomposition
$
K=\Z o\oplus F^1 K
$.
In other words every element $v\in K$ can be written as
$
v=r_v {{o}}+v^1
$
for a unique integer $r_v\in \Z$ and  unique $v^1\in F^1V$. We define the rank and degree functions $r,d:K\rightarrow \mathbb{Z}$ as
\begin{align*}
r(v)&\define r_v\\
d(v)&\define (v^1,\omega)=-\langle v^1,\tilde{\omega}\rangle=-\langle v,\tilde{\omega}\rangle 
\end{align*}
.\\
If $r_v\neq 0$ (i.e.\ $v\not\in F^1K$) we  put $\tilde{\eta}_v=\frac{1}{r_v}v^1 \in K_{\Q}$ and we let $\eta_v\in \Num_{\Q}(K)$ be the image of $\tilde{\eta_v}$.
\begin{lemmas}
The morphisms $d,r:K \mor \Z$ are linear. $r$ is independent of the choice of ${{o}}$, up to sign.
$d$ is determined up to sign and a multiple of $r$. 
The partially defined map $K\times K\rightarrow \Num(K)_{\Q}:(v,w) \mapsto \eta_v-\eta_w$ is determined up to a sign.
\end{lemmas}
\begin{proof}
These are easy verifications.
\end{proof}
The rank and degree functions are connected to the anti-symmetrization of $\langle-,-\rangle$. 
\begin{propositions} 
\label{prop:anti} 
Let $\{v,w\}\define \langle v,w\rangle-\langle w,v\rangle$ be the anti-symmetrization
of $\langle-,-\rangle$. Then
\begin{equation}
\label{eq:anti}
\{v,w\}=
\det\begin{bmatrix}
d(v)&d(w)\\
r(v)&r(w)
\end{bmatrix}
\end{equation}
and if $r(v)\neq 0$, $r(w)\neq 0$.
\begin{equation}
\label{eq:anti2}
\{v,w\}=r(v)r(w)(\eta_v-\eta_w,\omega)
\end{equation}
\end{propositions}
\begin{proof} Since $\langle-,-\rangle$ is symmetric on $F^1K$, $\{-,-\}$ is zero on $F^1K$. So we only have to
consider the case $v={{o}}$ and $w\in F^1K$. But then
\begin{align*}
\{{{o}},w\}&= \langle {{o}},w\rangle-\langle w,{{o}}\rangle\\
&=\langle w,(s-1){{o}}\rangle\\
&=-(\omega,w)\\
&=-d(w)
\end{align*}
which is indeed equal to the righthand side of \eqref{eq:anti}. To verify \eqref{eq:anti2} we note that 
\begin{align*}
d(v)r(w)-d(w)r(v)&=
(\eta_v,\omega)r(v)r(w)-(\eta_w,\omega)r(w)r(v)\\
&=r(v)r(w)(\eta_v-\eta_w,\omega)\qed
\end{align*}
\def\qed{}\end{proof}
Our definitions coincide with the usual notions in the geometric case.
\begin{lemmas}\label{lem:rank-degree}
Let $X$ be a smooth projective surface such that $\delta(X)\neq 0$. Put $K=K(X)_\text{num}$. 
Let ${{o}}=[\r{O}_X]$. Then for a coherent sheaf $ \r{F}$ on $X$, we have
\begin{enumerate}
\item
$r(\r{F})=\rk(\r{F})$.
\item
\label{item:gen}
$d(\r{F})=c_1(\r{F})\cdot c_1(\omega_X)$.
\end{enumerate}
Moreover
\begin{enumerate}
\setcounter{enumi}{2}
\item
\label{item:delta}
$\delta(K)=\delta(X)$.
\end{enumerate}
\end{lemmas}

\begin{proof}
The functions $\rk$ and $r$ are both zero on $F^1K$ and satisfy $1=\rk([\r{O}_X])=r([\r{O}_X])$. It follows that they must coincide.

If is sufficient to prove \eqref{item:gen} on generators of $K(X)_{\num}$. Hence we may assume that $\r{F}$ is a line bundle. In this case, since the difference of two line bundles always lies in $F^1(K)$, the decomposition from \S \ref{sec:rank} of $[\r{F}]\in K$ takes the form $[\r{F}]=[\r{O}_X]+([\r{F}]-[\r{O}_X])$
so that $[\r{F}]^1=([\r{F}]-[\r{O}_X])$. Now, the canonical element $\tilde{\omega}$ associated to $o$ is given by
\[
\tilde{\omega}=(s-1){{o}}=(s-1)[\r{O}_X]=[\omega_X]-[\r{O}_X]
\]
Hence $d(\r{F})=([\r{F}]^1,\tilde{\omega})=-\langle [\r{F}]-[\r{O}_X],[\omega_X]-[\r{O}_X]\rangle=
c_1(\r{F})\cdot c_1(\omega_X)$, using Lemma  \ref{lem:eulerform-picardform}.
Finally, since $\delta(K)=d(\omega_X)$, \eqref{item:delta} follows from (\ref{item:gen}).
\end{proof}
\section{Exceptional bases and mutation}
\label{sec:exmut}  
We recall some standard facts about mutation. See e.g.\ \cite{BondalSymplectic}. 
Throughout $K=(K,\langle-,-\rangle,s)$ is a Serre lattice of rank $n$ as defined in \S\ref{def:lattice}.
\begin{definition}\label{def:exceptionalbasis}
An element $e \in K$ is exceptional if $ \bl{e}{e}=1$. An exceptional pair in $K$ is a pair of exceptional elements $(v,w)\in K\times K$
such that $\langle w,v\rangle=0$.
A basis $(e_1, \ldots e_n)$ for $K$ is  exceptional if the $e_i$ are exceptional and $\bl{e_i}{e_j}=0$ for $j<i$.
A helix in $K$ is a sequence $(e_i)_{i \in \Z}$ such that  $\forall k: e_{k+n}=s^{-1}e_k$ and such that $(e_1,\ldots,e_n)$ is an exceptional basis.
\end{definition}
Every exceptional basis can be extended to a helix.
It is easy to see that every ``thread'' $(e_k,\ldots,e_{k+n-1})$ in a helix is an exceptional basis. 
\begin{definition}\label{def:braidgroupactions2}
If $(v,w)$ is an exceptional pair, the left mutation of $(v,w)$ is defined as 
\[
\sigma(v,w)\stackrel{\textrm{def}}{=}(w-\bl{v}{w}v,v).
\]
For an exceptional basis $E=(e_1, \ldots , e_n)$,  and $i \in \{1\,\ldots ,n-1\}$ we define 
 the left mutation at $i$ as
$$\sigma_i(E)\stackrel{\textrm{def}}{=} (e_1, \ldots,\sigma(e_i,e_{i+1}), \ldots, e_n)$$
\end{definition}
It is well known that the $(\sigma_i)_{\{1,\ldots , n-1\}}$ define an action of the braid group $B_n$ on the set of exceptional bases
in $K$ \cite{BondalSymplectic}.  If $E=(e_1,\ldots,e_n)$ is an exceptional basis for $K$ with corresponding
helix $H=(e_i)_{i\in \d{Z}}$ then we denote by $\rho(H)$  the right shift $(e_{i+1})_{i \in \d{Z}}$ of $H$. Looking at the initial  thread yields an operation on $E$ given by
$
\rho(E)=(se_n,e_1,\ldots,e_{n-1})
$.
 One checks that in fact
\[\rho(E)=\sigma_1\cdots\sigma_{n-1}(E).\] In particular $\rho(E)$ is contained in the $B_n$-orbit of $E$.

If we put $\rho=\sigma_1\cdots\sigma_{n-1}\in B_n$ then we have
$\sigma_{i}=\rho^{i-1}\sigma_1\rho^{-i+1}$ for $i=1,\dots,n-1$.
This allows us to extend the definition of $\sigma_i$ for any $i \in \d{Z}$ by
\[
\sigma_{i}=\rho^{i-1}\sigma_1\rho^{-i+1}, \,\,\forall i \in \d{Z}.
\]

The elements $(\sigma_i)_{i\in \Z}\in B_n$ act
naturally on helices changing $e_{i+kn}$ and $e_{i+kn+1}$ for $k\in \Z$. If $E=(e_1,\ldots,e_n)$ is an exceptional basis then so is 
\[
\epsilon_i(E)\define(e_1,\ldots,e_{i-1},-e_i,e_{i+1},\ldots,e_n)
\]
This shows that the set of exceptional bases actually admits an action of the \emph{signed} braid group $\Sigma B_n\define B_n\# (\Z/2\Z)^n$
where $B_n$ acts on $(\Z/2\Z)^n$ through its quotient~$S_n$ by the pure braid group. The operators $\epsilon_i$ also act on
helices $H=(e_j)_{j\in \d{Z}}$, changing the sign of the elements $(e_{i+kn})_k$.
Finally, if $M$ is the Gram matrix of $E$ and $\sigma \in \Sigma B_n$, we denote $\sigma(M)$ the Gram matrix of $\sigma(E)$. In this way we also obtain an action of $\Sigma B_n$ on the set of exceptional matrices. As some of the verifications below are best done by computer
we record the well-known formulas for this action: For $k<l$ we have
\[
\sigma_i(M)_{kl}=
\begin{cases}
M_{kl}&\text{if $\{k,l\}\cap \{i,i+1\}=\emptyset$}\\
M_{k,i+1}-M_{i,i+1}M_{ki}&\text{if $l=i$}\\
M_{k,i}&\text{if $l=i+1,k\neq i$}\\
M_{i+1,l}-M_{i,i+1}M_{il}&\text{if $k=i,l\neq i+1$}\\
M_{i,l}&\text{if $k=i+1$}\\
-M_{i,i+1}&\text{if $k=i$, $l=i+1$}
\end{cases}
\]
\[
\sigma_i^{-1}(M)_{kl}=
\begin{cases}
M_{kl}&\text{if $\{k,l\}\cap \{i,i+1\}=\emptyset$}\\
M_{k,i+1}&\text{if $l=i$}\\
M_{k,i}-M_{i,i+1}M_{k,i+1}&\text{if $l=i+1,k\neq i$}\\
M_{i+1,l}&\text{if $k=i,l\neq i+1$}\\
M_{il}-M_{i,i+1}M_{i+1,l}&\text{if $k=i+1$}\\
-M_{i,i+1}&\text{if $k=i$, $l=i+1$}
\end{cases}
\]
\[
\epsilon_i(M)_{kl}=
\begin{cases}
-M_{kl}&\text{if $i\in\{k,l\}$}\\
M_{kl}&\text{otherwise}
\end{cases}
\]
Assume that $K$, $K'$ are Serre lattices with exceptional bases
$(e_i)_i$, $(e'_i)_i$. Below it will be convenient to introduce the
notation $(K,(e_i)_i)\overset{\Sigma B_n}{\cong} (K',(e'_i)_i)$  to indicate that there is an isomorphism of
Serre lattices $\phi:K\rightarrow K'$ such that $\phi(e'_i)_i$ is in the $\Sigma B_n$ orbit of $(e'_i)_i$.
\section{Numerical blowing up/down}
\subsection{More general codimension filtrations}
We have defined the codimension filtration for a Serre lattice $K$ which is of
surface* type. However in this section we will discuss numerical analogues of blowing up and blowing down. These procedures will change the degree of $K$, which in particular may become zero. Assuming that $K$ is of surface* type is thus not very natural in this context. Therefore we introduce 
a generalized version of the codimension filtration:
\begin{definitions}
Let $K=(K,\langle-,-\rangle,s)$ be a lattice of rank $n$ of surface type and put $V=K_{\Q}$. A codimension filtration
on $V$ is a filtration
\[
0=F^3V\subset F^2V\subset F^1V\subset F^0V=V
\]
such that 
\begin{enumerate}
\item \label{cod1} $(s-1)F^iV\subset F^{i+1}V$.
\item \label{cod2} $\dim F^1V=n-1$, $\dim F^2V=1$.
\item \label{cod3} $\langle F^1V,F^2 V\rangle=0$.
\end{enumerate}
\end{definitions}
Note that \eqref{cod1} implies that the $F^i V$ are $s$-invariant. 
It follows from Lemma \ref{filtrationunique} that if $K$ is of surface* type then a codimension filtration exists
and coindices with the filtration defined in Lemma \ref{filtration}. In particular it is also unique.

A codimension filtration $(F^i V)_i$ is determined by its induced
filtration $F^iK=K\cap F^iV$ on $K$. We will  refer to the latter also
as a codimension filtration. The results and definitions introduced in
the surface* case remain valid for arbitrary codimension filtrations. We will however decorate our notations with an
index $F$, indicating the choice of a codimension filtration (which can only be a real choice if $\delta(K)=0$).
\subsection{Numerically blowing up}\label{subsec:nb}
Let $K=(K,\langle-,-\rangle,s)$ be of surface type with a codimension filtration $(F^i K)_i$. Let $z\in F^2 K$ (in particular $sz=z$). Then the blowup $\tilde{K}=(\tilde{K},\langle-,-\rangle,\tilde{s})$ of $K$ in $z$ is defined as follows.
Put $\tilde{K}\define \Z f\oplus K$ and extend $\langle-,-\rangle$ to $\tilde{K}$ via: 
\begin{align*}
\langle f,f\rangle&=1\\
\langle -,f\rangle&=0\\
\langle f,y\rangle&=\langle z,y\rangle
\end{align*} 
One checks that $\tilde{K}$ has a Serre automorphism $\tilde{s}$ given by
\begin{align*}
\tilde{s}y&=sy-\langle y,z\rangle f\qquad \text{for $y\in K$}\\
\tilde{s}f&=f+z
\end{align*}
and furthermore that the codimension filtration $F$ on $K$ extends to a codimension filtration $\tilde{F}$ on $\tilde{K}$ via
\begin{align*}
F^1\tilde{K}&=F^1K\oplus \Z f\\
F^2\tilde{K}&=F^2K
\end{align*}
We immediately see that we have an orthogonal decomposition
\[
\Num_F(\tilde{K})\define \big(F^1K\ds \d{Z} F\big)/F^2K\cong \d{Z} F\ds F^1K/F^2K\define\Z f\oplus \Num_F(K)
\]
\begin{lemmas} \label{deltablowup} We have
$
\delta_F(\tilde{K})=\delta_F(K)-\langle o,z\rangle^2
$.
\end{lemmas}
\begin{proof}
Since there is a collision of notation with the use of 
$\tilde{\omega}$ we will temporarily write $\omega_F\in F^1K$ for  the canonical element of $K$.
Clearly if $o\in K$ is a structure element in $K$ following Definition \ref{def:delta} then it remains one in $\tilde{K}$. We compute the canonical element in $\tilde{K}$: 
\[
\tilde{\omega}_F=(\tilde{s}-1)o=\omega_F-\langle o,z\rangle f
\]
Thus
\begin{align*}
\langle \tilde{\omega}_F,\tilde{\omega}_F\rangle&=\langle\omega_F,\omega_F\rangle-\langle o,z\rangle \langle z,\omega_F\rangle+\langle o,z\rangle^2\\
&=\langle\omega_F,\omega_F\rangle+\langle o,z\rangle^2
\end{align*}
where we have used
\[
\langle z,\omega_F\rangle=\langle z,(s-1)o\rangle=\langle (s^{-1}-1) z,o\rangle=0\qed
\]
\def\qed{}\end{proof}
Not that if $(e_1,\ldots,e_n)$ is exceptional basis with Gram matrix $M$ then $(f,e_1,\ldots,e_n)$ is an exceptional
basis for $\tilde{K}$ and the corresponding Gram matrix $\tilde{M}$ is given by
\[
\tilde{M}=
\begin{bmatrix}
1 &\langle z,e_1\rangle&\cdots&\langle z,e_n\rangle\\
0 &\\
\vdots && M\\
0&
\end{bmatrix}
\]
\subsection{Numerical blowup of \mathversion{bold} $\mathbb{P}^2$}
\label{sec:blowupP2}
Let $K\define K(\mathbb{P}^2)$. Then $K$ is equipped with the exceptional basis $e_2,e_3,e_4$ coming from the 
exceptional sequence \[\bigg(\mathcal{O}(-1), \Omega(1),\mathcal{O}\bigg).\]
Let $x\in \mathbb{P}^2$. We  have $F^2K=\Z [\mathcal{O}_x]$ and so it is possible to perform a numerical
blowup of $K$ at $z=n[\mathcal{O}_x]$, $n\in \Z$.
By a sign change we may assume $n\ge 0$.
Below we let $K_n$ be the blowup of $K$ at $n[\mathcal{O}_x]$ for $n\ge 0$. We will always implicitly assume that $K_n$
is equipped with the exceptional basis $(f,e_2,e_3,e_4)$. The Gram matrices of $K_n$ are as
in the second series of solutions in Theorem \ref{mainth}.  It follows by Lemma \ref{deltablowup} 
that indeed $\delta(K_n)=9-n^2$, as claimed in the introduction.
\subsection{Numerically blowing down}
\label{blowingdown}
It turns out the numerical blowup construction from \S \ref{subsec:nb} is reversible:
\begin{lemma}\label{lem:blowdown}
Let $K=(K,\langle-,-\rangle,s)$ be a lattice of surface type equipped with
a codimension filtration and $f \in F^1K$. Define $\bar{K}=({}^\perp f,\bl{-}{-})$ and let
\begin{align*}
\bar{s}y&=sy+\langle y,z\rangle f\qquad \text{for $y\in \bar{K}$}\\
F^1\bar{K}&=F^1K\cap \bar{K}\\
F^2\bar{K}&=F^2K
\end{align*}
Then $(\bar{K},\bl{-}{-},\bar{s})$ is a lattice of surface type with codimension filtration $F^i\bar{K}$. Moreover, the numerical blowup of $\bar{K}$at $z\define (\bar{s}-1)f$ is precisely $K$.\qed
\end{lemma}

\section{The case of rank {\boldmath $4$}}
\label{sec:rk4}
In this section we give the proof Theorem \ref{mainth}. To this end we recall some notation introduced in \S \ref{sec:introduction}. We will fix a lattice of surface type $(K,\bl{-}{-},s)$ and assume it has an exceptional basis $E\define(e_1,e_2,e_3,e_4)$ with Gram matrix
\[
M=
\begin{bmatrix}
1&a&b&c\\
0&1&d&e\\
0&0&1&f\\
0&0&0&1
\end{bmatrix}
\]
We recall from formula \ref{eq:sformula} that the matrix for $s$ is given by $s=M^{-1}M^{\mathrm{t}}$ and that by (\ref{eq:rk4}), the unipotency of $s$ translates into 
\begin{equation}
\label{eq:rk4new}
\begin{cases}
acdf-abd-ace-bcf-def+a^2+b^2+c^2+d^2+e^2+f^2=0\\
af-be+cd=0
\end{cases}
\end{equation}
 We will prove Theorem \ref{mainth} in several steps:
\begin{enumerate}
\item We first show that any Serre lattice $K$ of rank $\le 4$ with unipotent Serre automorphism is of surface type.
\item We treat the case $\delta(K)=0$ through an adhoc argument starting directly from the equations \eqref{eq:rk4new}, reducing us to the case where $K$ is of surface* type.
\item We show that by performing appropriate mutations on an exceptional basis $(e_1,e_2,e_3,e_4)$ we may always reduce to to one of the following situations: (1.) $e_1$ has rank zere or  (2.) $\langle e_1,e_2\rangle=2=\langle e_3,e_4\rangle=2$.
\item By \S\ref{sec:blowupP2}, Case (1.) corresponds to a numerical blowup of $K(\mathbb{P}^2)$.
\item Case (2.) is treated again through an adhoc argument starting from \eqref{eq:rk4new}.
\item Finally we check that all solutions are different, mostly relying on the value of $\delta(K)$.
\end{enumerate}
\subsection{Preliminaries}
We first note that in low rank the second condition in Definition \ref{def:SPStype} is redundant
and hence in particular the lattice $K$ in Theorem \ref{mainth} is of surface type.
\begin{lemmas}\label{lem:SPStyperank4}
Let $K=(K,\langle-,-\rangle,s)$ be a Serre lattice of rank $\le 4$. If $s$ is unipotent then
$K$ is of surface type.
\end{lemmas}
\begin{proof}
We need to show that $\rk(s-1)\le 2$ on $K$. We will work in the $\mathbb{Q}$-vector space $V=K_{\mathbb{Q}}$.
Since we have $$\{v, w\}\stackrel{\textrm{def}}{=}\bl{w}{v}-\bl{v}{w}=\bl{v}{(s-1)w}$$ it follows immediately that  $\ker(s-1)$ is the radical of $\{-,-\}$.
 Hence $\im(s-1)\cong  V/\ker(s-1)$ is endowed with a nondegenerate antisymmetric form and hence must be even dimensional. As $(s-1)$ is nilpotent, it cannot be surjective and 
therefore $\dim\im(s-1)\neq 4$. It follows that $\dim\im(s-1)\le 2$.
\end{proof}
\subsection{The degree zero case} 
We now dispense with the case that $K$ is of surface but not of surface* type, or put differently the case 
$\delta(K)=0$. To this end recall from \S \ref{sec:blowupP2} that $K_3$ denotes the numerical blowup of $K(\d{P}^2)$ at $3[\r{O}_x]$ for $x\in \d{P}^2$
\begin{lemmas}
Let $K=(K,\langle-,-\rangle,s)$ be a Serre lattice of rank 4 with unipotent Serre automorphism $s$ and an exceptional 
basis $(e_i)_i$. If $(s-1)^2=0$ then $(K,(e_i)_i)\overset{\Sigma B_4}{\cong} K_3$.
\end{lemmas}
\begin{proof}
Consider the quadratic form \[K\times K\rightarrow \Z:(v,w)\mapsto -\langle (s-1)v,(s-1)w\rangle\]  Its matrix is given
by $F=-(MM^{-t}-1)M(M^{-1}M^{\mathrm{t}}-1)$. Since $(s-1)^2=0$, and
\[
-\langle (s-1)v,(s-1)w\rangle=-\bl{v}{s^{-1}(s-1)^2w}=0
\] we conclude $F=0$. Using a computer algebra system, we find
\begin{align*}
0=F_{00}&=-a c d f + a b d + a c e + b c f - a^{2} - b^{2} - c^{2}\\
0=F_{33}&=-a c d f + a c e + b c f + d e f - c^{2} - e^{2} - f^{2}
\end{align*}
Let $q_1$ be the lefthand side of the first equation in \eqref{eq:rk4new}. Adding $q_1$ to $F_{00}$ and $F_{33}$ we see that 
\begin{equation}
\label{eq:second}
\begin{aligned}
d^{2} + e^{2} + f^{2}-def&=0\\
 a^{2} + b^{2} + d^{2}-a b d&=0
\end{aligned}
\end{equation}
These are ordinary Markov equations as in \eqref{eq:rk3}. As explained in the introduction, the braid group $B_3$ acts transitively on the set of its solutions so that after applying mutations $\sigma_2,\sigma_3$ we may assume that $d=e=f=3$.
Substituting this in \eqref{eq:rk4new} we obtain
\begin{align*}
a^{2} - 3 a b + 6 a c + b^{2} - 3 b c + c^{2}&=0\\
a - b + c&=0
\end{align*}
This system can be easily solved by substitution and we obtain $b=2a$, $c=a$. Substituting this in the second equation in \eqref{eq:second} (together
with $d=3$) yields $9-a^2=0$. Hence $a=\pm 3$. The case $a=-3$ is obtained from $a=3$ by applying $\epsilon_1$. The Gram matrix becomes
\[
\begin{bmatrix}
1&3&6&3\\
0&1&3&3\\
0&0&1&3\\
0&0&0&1
\end{bmatrix}
\]
which has the required form.
\end{proof}
\subsection{Reducing exceptional pairs} 
We can henceforth assume that the lattice $K=(K,\langle-,-\rangle,s)$ is of surface* type. We fix a structure element
${{o}}\in K$ and use the associated notations as in \S\ref{sec:numer}.
For a collection of elements $(v_1, \ldots , v_n)$ in $K$, we introduce the following invariant.
\begin{equation}
\label{eq:Markovnumber}
\r{M}(v_1, \ldots , v_n )\stackrel{\textrm{def}}{=}\sum_1^n \vert r(v_i) \vert.
\end{equation}
We view $\r{M}$ as a measure for the complexity of $(v_1, \ldots , v_n)$. The next lemma is our main technical tool. 
\begin{lemmas} \cite{Rudakov} \label{rankreduction}
Let $(v,w)$ be an exceptional pair with $r(v),r(w)>0$.
Assume the following two conditions are satisfied 
\begin{itemize}
\item
$(\eta_w-\eta_v, \eta_w-\eta_v)<0$.
\item
$(\eta_w-\eta_v,\omega)<0$.
\end{itemize}
(See \S \ref{sec:rank} for notations.)
Then either $\r{M}(\sigma(v,w))<\r{M}(v,w)$ or  $\r{M}(\sigma^{-1}(v,w))<\r{M}(v,w)$ where $\sigma$ is as in Definition \ref{def:braidgroupactions2}. 
\end{lemmas}

\begin{proof}
By Proposition \ref{prop:anti}, 
\[
h\stackrel{\textrm{def}}{=}\bl{v}{w}=r(v)r(w)(\eta_v-\eta_w,\omega)>0.
\]
We compute
\begin{equation}
\label{eq:rankreduction}
\begin{aligned}
0>(\eta_w-\eta_v,\eta_w-\eta_v)
&=-\langle w/r(w)-v/r(v),w/r(w)-v/r(w)\rangle\\
&=-\frac{1}{r(v)^2}+\frac{h}{r(v)r(w)}-\frac{1}{r(w)^2}\\
&=-\frac{1}{r(v) r(w)}\big(r(v)^2-hr(v)r(w)+r(w)^2\big)
\end{aligned}
\end{equation}
Consider the quadratic form
$$Q: \Q^2 \mor \Q: (x,y) \fun x^2-hxy+y^2$$
Let $[-,-]$ denote the associated symmetric bilinear form. That is,
$$[--]: \Q^2\times \Q^2: \Q: ((x,y),(a,b)) \mor ax+yb-\frac{h}{2}(ay+xb)$$
Then \eqref{eq:rankreduction} becomes
\begin{equation}
\label{eq:Q}
\begin{aligned}
0<Q(r(v),r(w))&=[(r(v),r(w)),(r(v),r(w))]\\
&= r(v)[(1,0),(r(v),r(w))]+r(w)[(0,1),(r(v),r(w))]
\end{aligned}
\end{equation}
It follows that one of the two terms on the right hand side is strictly positive. Let's first assume this is true for the last one.
Then 
\[
0<[(0,1),(r(v),r(w))]=r(w)-\frac{h}{2}r(v)
\]
which yields $r(w)-hr(v)>-r(w)$. Since we trivially also have $ r(w)-hr(v)<r(w)$, we conclude $\vert r(w)-hr(v)\vert <\vert r(w) \vert $. Hence
\[
\r{M}(\sigma(v,w))= \vert r(w)-hr(v)\vert+\vert r(v) \vert < \vert r(w) \vert +\vert r(v) \vert =\r{M}((v,w))
\]
If it is the first term in the right hand side of \eqref{eq:Q} which is strictly positive then we obtain a similar conclusion but now with $\sigma^{-1}$ instead of $\sigma$.
\end{proof}
We will also use the following variant of this lemma:
\begin{lemmas}\label{rankreduction2}
Let $(v,w)$ be an exceptional pair with $r(v),r(w)>0$.
Assume the following two conditions are satisfied 
\begin{itemize}
\item
$(\eta_w-\eta_v, \eta_w-\eta_v)=0$.
\item
$(\eta_w-\eta_v,\omega)<0$.
\end{itemize}
Then either $\r{M}(\sigma(v,w))<\r{M}(v,w)$ or  $\r{M}(\sigma^{-1}(v,w))<\r{M}(v,w)$ or else $\langle v,w\rangle=2$ and $r(v)=r(w)$.
\end{lemmas}
\begin{proof} We keep the notation of the proof of lemma \ref{rankreduction}. We now have to deal with the additional possibility that the two expressions $[(1,0),(r(v),r(w))]$ and $[(0,1),(r(v),r(w))]$ are zero. This can only happen
when $Q$ is degenerate. I.e. when $h=\langle v,w\rangle=2$. In that case $(r(v),r(w))$ is in the radical of $Q$ which implies $r(v)=r(w)$.
\end{proof}

\subsection{Some auxillary results on plane geometry}
In this section we state some adhoc results which will be used below. 
\begin{lemmas}
\label{lem:indefinite}
Assume that $(-,-)$ is a non-degenerate symmetric bilinear form on a two-dimensional real vector space $H$.
Assume that there are vectors $(T_i)_{i\in \Z}\in H$ satisfying
\begin{enumerate}
\item $T_{i+4}=T_i$.
\item \label{x2} $(T_i,T_{i+2})=0$.
\item \label{x3} $(T_i,T_{i+1})>0$.
\end{enumerate}
Then $(-,-)$ is indefinite.
\end{lemmas}
\begin{proof}
We argue by contradiction. Note that $T_i\neq 0$ for all $i$ by \eqref{x3}. Assume first that 
$(-,-)$ is positive definite. Since by \eqref{x3} we have $(T_1,T_2)>0$, $(T_3,T_2)>0$, $(T_1,T_4)=(T_5,T_4)>0$, $(T_3,T_4)>0$ we see that $T_2$, $T_4$
are both in the interior of the quadrant spanned by the orthogonal vectors $T_1,T_3$. But then it is 
clear that the $T_2,T_4$ cannot be orthogonal among themselves, contradicting the hypothesis.

Now assume that $(-,-)$ is negative definite. Put $[-,-]=-(-,-)$. Then $[-,-]$ is positive
definite but $[T_i,T_{i+1}]<0$. We fix this by replacing $(T_i)_i$ by $(T'_i)_i$ with
\[
T'_{i}
=
\begin{cases}
T_i&\text{if $i$ is even}\\
-T_i&\text{if $i$ is odd}
\end{cases}
\]
Now we argue as above with $[-,-]$ replacing $(-,-)$.
\end{proof}

\begin{lemmas}
\label{lem:sequence}
Assume that $(-,-)$ is a non-degenerate symmetric bilinear form on a two-dimensional real vector space $H$.
Assume that there are vectors $\omega,(T_i)_{i\in \Z}\in H$ satisfying
\begin{enumerate}
\item \label{c1} $T_{i+4}=T_i$.
\item \label{c2} $\omega\neq 0$.
\item  \label{c3}
$(T_i,T_{i+2})=0$.
\item \label{c4}
$(T_i,T_{i+1})>0$.
\item \label{c5}
$(T_i,T_i)\ge 0\Rightarrow (T_i,\omega)<0$.
\end{enumerate}
Then one of the following is true
\begin{enumerate}
\item There exists an $i$ such that $(T_i,T_i)=(T_{i+2},T_{i+2})=0$ (and hence by \eqref{c5}: $(T_i,\omega)<0$, $(T_{i+2},\omega)<0$).
\item There exists an $i$ such that $(T_i,T_i)<0$ and $(T_i,\omega)<0$. 
\end{enumerate}
\end{lemmas}
\begin{proof}
By Lemma \ref{lem:indefinite} we know that $(-,-)$ must be indefinite.  Moreover if (\ref{c1}-\ref{c5}) hold for $T_1,T_2,T_3,T_4$ then they also hold for
$T_1,T_4,T_3,T_2$. Note also that \eqref{c4} implies that $T_i\neq 0$ for all $i$.

Assume the conclusion of the lemma is false. Then by \eqref{c5} we have for all $i$:
\begin{equation}
\label{blah}
(T_i,T_i)\ge 0\text{ or }(T_i,\omega)\ge 0\text{ but not both }
\end{equation}
and moreover there is at least one even and one odd $i$ for which $(T_i,T_i)\neq 0$. We will obtain a contradiction.
By shifting $(T_i)_i$ we may assume that either
$(T_1,T_1)>0$ or $(T_3,T_3)<0$. Since $(T_1,T_3)=0$ and $(-,-)$ is indefinite and non-degenerate we obtain
$(T_1,T_1)> 0$ and $(T_3,T_3)< 0$ and moreover $\{T_1,T_3\}$ forms a basis for $H$.
A similar reasoning for $T_2,T_4$ (possibly after exchanging them) yields $(T_2,T_2)>0$, $(T_4,T_4)<0$.
Hence by \eqref{blah} we obtain $(T_1,\omega)<0$, $(T_3,\omega)\ge 0$, $(T_4,\omega)\ge 0$.

Write 
$T_4=\gamma T_1+\delta T_3$. 
Expressing 
$(T_1,T_4)=(T_5,T_4)>0$, $(T_3,T_4)>0$ yields
$\gamma>0$, $\delta<0$.
Applying $(-,\omega)$ to
$T_4=\gamma T_1+\delta T_3$ yields a contradiction.
\end{proof}
\subsection{Minimal forms for exceptional bases}
We now let $(e_i)_{i\in\Z}$ be the helix associated to the exceptional basis $E$ as in \S\ref{sec:exmut}. 
To simplify notation we write $r_i=r(e_i)$, $\eta_i=\eta_{e_i}$, etc\dots.
Note that $r_{4+i}=r_i$ as $e_{i+4}=s^{-1}e_i$. \label{perlingperling}
\begin{lemmas} \label{lem:cases}
 By acting through an element of  $\Sigma B_4$ we may assume that one of the following conditions holds
\begin{align*}
\text{Case 1:}& &r_1&=0\\
\text{Case 2:}& &\langle e_1,e_2\rangle =2&\text{ and }\langle e_3,e_4\rangle=2
\end{align*}
\end{lemmas}
\begin{proof} Since $\r{M}$ takes values in $\d{N}$, we may replace $E$ by a basis in its $\Sigma B_4$-orbit
  such that $\mathcal{M}(E)$ is minimal. If there exists an $i$ such
  that $r_i=0$ then we are done (after applying a rotation $\rho^{i-1}$ to $E$, see \S\ref{sec:exmut}). So we assume
  $r_i\neq 0$ for all $i$.  Applying appropriate sign changes $\epsilon_i$ we may assume $r_i> 0$ for all $i$. 
Put $T_i=\eta_{i+1}-\eta_i$. 
We verify the conditions for Lemma \ref{lem:sequence} on $H\define \Num_\d{R}(K)$
\begin{enumerate}
\item \label{cond10} $T_{i+4}=T_i$. This follows from the fact that $\tilde{\eta}_{i+5}-\tilde{\eta}_{i+4}
= e_{i+5}/r_{i+5}-e_{i+4}/r_{i+4}=s^{-1}(e_{i+1}/r_{i+1}-e_i/r_i)=\tilde{\eta}_{i+1}-\tilde{\eta}_i$.
\item \label{cond11} $\omega\neq 0$. This follows from the fact that $K$ is of surface* type by Lemma
\ref{lem:deltawd}.
\item \label{cond1} $(T_i,T_{i+2})=0$. We have
\begin{align*}
( T_i,T_{i+2})= (T_{i+2},T_i)&=(\eta_{i+3}-\eta_{i+2},\eta_{i+1}-\eta_i)\\
=&\left(\frac{e_{i+3}}{r_{i+3}}-\frac{e_{i+2}}{r_{i+2}},\frac{e_{i+1}}{r_{i+1}}-\frac{e_{i}}{r_{i}}\right)\\
=&0.
\end{align*}
\item \label{cond2} $(T_i,T_{i+1})>0$. A similar computation as in $(3)$ shows in fact that $(T_i,T_{i+1})=1/r^2_{i+1}$
\item \label{cond3} $(T_i,T_i)\ge 0\Rightarrow (T_i,\omega)<0$. The sum $T_i+T_{i+1}+T_{i+2}+T_{i+3}$ is equal
to $\eta_{i+4}-\eta_i$ and one computes 
\begin{align*}
\tilde{\eta}_{i+4}-\tilde{\eta}_i&=e_{i+4}/r_{i+4}-e_i/r_i \\
&=(s^{-1}-1)(e_i)/r_i=s^{-1}(s-1)(\tilde{\eta}_i+o)=s^{-1}(s-1)(\tilde{\eta}_i)+s^{-1}(s-1)(o)
\end{align*}
Hence modulo $F^2K_{\mathbb{R}}$:
\[
\eta_{i+4}-\eta_i=-s^{-1}(s-1)(o)=-\omega
\]
Hence $\omega=-(T_i+T_{i+1}+T_{i+2}+T_{i+3})$ and thus $(\omega,T_i)=-(T_i,T_i)-(T_i,T_{i+1})-(T_i,T_{i-1})$. It now suffices to apply (\ref{cond1}),(\ref{cond2}).
\end{enumerate}
From Lemma \ref{lem:sequence} we deduce that after rotating, one of the following conditions holds.
\begin{enumerate}
\item \label{z1} $(T_1,T_1)=(T_{3},T_{3})=0$, $(T_1,\omega)<0$ and $(T_3,\omega)<0$.
\item \label{z2} $(T_1,T_1)<0$ and $(T_1,\omega)<0$. 
\end{enumerate}
However \eqref{z2} contradicts the minimality $E$ using Lemma \ref{rankreduction}. Similarly if $\langle e_1,e_2\rangle\neq 2$ or
$\langle e_3,e_4\rangle \neq 2$ then \eqref{z1} also contradicts the minimality of $E$, using Lemma \ref{rankreduction2}.
Hence we are done.
\end{proof}
\subsection{Case 1}
Now we discuss the two minimal cases exhibited in Lemma \ref{lem:cases} individually. We keep the same notations. We first assume $r_1=0$. This is equivalent to $e_1 \in F^1(K)$. Applying Lemma \ref{lem:blowdown}, we conclude that $K$ is obtained by numerical blowup of the sublattice $K'={}^\perp e_1=\langle e_2,e_3,e_4\rangle$ at 
some $z\in F^2 K'$. Since $\rk K'=3$ we have $K'=K(\mathbb{P}^2)$ as explained in the introduction. 
It follows that $(K,(e_i)_i)\overset{\Sigma B_3}{\cong} K_n$ for $n\ge 0$ by 
\S\ref{sec:blowupP2}. 
\subsection{Case 2}
Now we assume $\langle e_1,e_2\rangle=2$, $\langle e_3,e_4\rangle=2$.
This means the Gram matrix has the form:
\[
M=
\begin{bmatrix}
1&2&b&c\\
0&1&d&e\\
0&0&1&2\\
0&0&0&1
\end{bmatrix}
\]
By possibly changing the sign of $e_1$, $e_2$ we may assume that $b\ge d$.
Substituting $a=f=2$ in  \eqref{eq:rk4new} we obtain
\begin{equation}
\label{eq:system}
\begin{aligned}
b^2 - 2bc - 2bd + c^2 + 4cd - 2ce + d^2 - 2de + e^2 + 8&=0\\
-be + cd + 4&=0
\end{aligned}
\end{equation}
Denoting the lefthand sides by $q_1$ and $q_2$ we find
\[
q_1-2q_2=(b - c - d + e)^2
\]
so that \eqref{eq:system} is equivalent to
\begin{align*}
-be + cd + 4&=0\\
b-c-d+e&=0
\end{align*}
Since it's easy to see that $(b,c,d,e)$ is a solution to this system if and only if $(b+t,c+t,d+t,e+t)$ is, we will classify the solutions assuming $d=0$. Then we must solve
\begin{align*}
be &=4\\
c&=b+e
\end{align*}
The solutions to this system are (taking into account $b\ge d=0$) 
\[
\begin{array}{|c c c|}\hline
b&c&e\\
\hline
2&4&2\\
\hline
1&5&4\\
\hline
4&5&1\\\hline
\end{array}
\]
We discuss these separately. For the solution $(2,4,2)$ one has $M=M_t$ with
\[
M_t=
\begin{bmatrix}
1&2&2+t&4+t\\
0&1&t&2+t\\
0&0&1&2\\
0&0&0&1
\end{bmatrix}
\]
One checks $ (\epsilon_1\sigma_1)(M_t)=M_{t+2} $ so that there are at 
most two orbits, respectively with representatives $M_0$ and 
$M_1$. The case $M_0$ corresonds to $\mathbb{P}^1\times
\mathbb{P}^1$ with its standard exceptional collection $\mathcal{O}(0,0),\mathcal{O}(1,0),\mathcal{O}(0,1),\mathcal{O}(1,1)$. It will be convenient to denote 
this lattice by $K'_1$.
So $(K,(e_i)_i)\overset{\Sigma B_3}{\cong} K'_1$.

 For the solution $M_1$ we note
\[
\epsilon_2\epsilon_4\sigma_1^{-1}\sigma_3 \sigma_2
\begin{bmatrix}
1&2&3&5\\
0&1&1&3\\
0&0&1&2\\
0&0&0&1
\end{bmatrix}
=
\begin{bmatrix}
1&1&2&1\\
0&1&3&3\\
0&0&1&3\\
0&0&0&1
\end{bmatrix}
\]
Hence $(K,(e_i)_i)\overset{\Sigma B_4}{\cong} K_1$.

\medskip

For the solution $(1,5,4)$ we have
\[
M_t=
\begin{bmatrix}
1&2&1+t&5+t\\
0&1&t&4+t\\
0&0&1&2\\
0&0&0&1
\end{bmatrix}
\]
and one checks that
$
(\epsilon_1\sigma_1)(M_t)=M_{t+1} 
$.
It follows that there is only a single orbit with representative $M_0$. We have
\begin{equation}
\label{eq:single}
\epsilon_4\sigma_1\sigma^{-1}_2\sigma_1\sigma_2
\begin{bmatrix}
1&2&1&5\\
0&1&0&4\\
0&0&1&2\\
0&0&0&1
\end{bmatrix}=
\begin{bmatrix}
1&2&4&2\\
0&1&3&3\\
0&0&1&3\\
0&0&0&1
\end{bmatrix}
\end{equation}
Hence in this case $(K,(e_i)_i)\overset{\Sigma B_3}{\cong} K_2$.

Finally for the solution $(4,5,1)$ we have
\[
M_t=
\begin{bmatrix}
1&2&4+t&5+t\\
0&1&t&1+t\\
0&0&1&2\\
0&0&0&1
\end{bmatrix}
\]
One checks
$
(\epsilon_3\sigma_2)^{-1}M_i=M_{i+1} 
$.
It follows  that there is again only a single orbit with representative $M_0$. This time we have
\[
\epsilon_1\epsilon_2\sigma_2\sigma_1\sigma_2\sigma_2\sigma_3^{-1}\sigma_2
\begin{bmatrix}
1&2&4&5\\
0&1&0&1\\
0&0&1&2\\
0&0&0&1
\end{bmatrix}
=
\begin{bmatrix}
1&2&4&2\\
0&1&3&3\\
0&0&1&3\\
0&0&0&1
\end{bmatrix}
\]
Hence also in this case $(K,(e_i)_i)\overset{\Sigma B_3}{\cong} K_2$.
\subsection{All solutions are different}
We have $\delta(K')=\delta(\mathbb{P}^1\times \mathbb{P}^1)=8$ and by
\S\ref{sec:blowupP2} $\delta(K_n)=9-n^2$. Hence the only possible non-trivial equivalence 
is between $K'_1$ and $K_1$ (the latter corresponds to $\mathbb{F}_1$).
One way to distinguish $K'_1$ and $K_1$ is to verify that $s\equiv 1\mod 2$
in the first case and $s\not\equiv 1\mod 2$ in the second case.

\addcontentsline{toc}{section}{References}


\providecommand{\bysame}{\leavevmode\hbox to3em{\hrulefill}\thinspace}
\providecommand{\MR}{\relax\ifhmode\unskip\space\fi MR }
\providecommand{\MRhref}[2]{%
  \href{http://www.ams.org/mathscinet-getitem?mr=#1}{#2}
}
\providecommand{\href}[2]{#2}

\end{document}